\documentclass[reqno,12pt]{amsart}
\usepackage{amsmath,amsthm,amsfonts,amscd}
\usepackage{eepic}
\usepackage{xy}
\usepackage{graphicx}
\usepackage{verbatim}
\usepackage{times}

\newtheorem{prop}{Proposition}[subsection]

\newtheorem{defn}[prop]{Definition}

\newtheorem{claim}{Claim}

\newtheorem{theorem}[subsubsection]{Theorem}

\newtheorem{assump}[subsubsection]{Assumption}

\newtheorem{proposition}[subsubsection]{Proposition}
\theoremstyle{definition}

\newtheorem{definition}[subsubsection]{Definition}

\theoremstyle{remark}

\theoremstyle{remark}

\newtheorem{remark}[subsubsection]{Remark}

\numberwithin{equation}{section}

\newcommand{\X}{\mathcal{X}}
\newcommand{\Y}{\mathcal{Y}}

\newcommand{\M}{\mathcal{M}}

\newcommand{\F}{\mathcal{F}}
\newcommand{\V}{\mathcal{V}}
\newcommand{\U}{\mathcal{U}}
\newcommand{\G}{\mathbb{G}}

\newcommand{\sI}{\mathcal{I}}

\newcommand{\sO}{\mathcal{O}}

\renewcommand{\O}{\mathcal{O}}

\newcommand{\XM}{\mathcal{X}\times \M_\L}
\newcommand{\T}{\overline{T}}
\newcommand{\g}{\overline{g}}
\newcommand{\f}{\overline{f}}

\newcommand{\XT}{\X \times T}
\newcommand{\gen}{\mathcal{E}}
\newcommand{\E}{\mathbb{E}}
\newcommand{\LM}{\mathbb{L}_{\M_\L}}
\newcommand{\LT}{\mathbb{L}_{T}}
\newcommand{\LXM}{\mathbb{L_{\X \times \M_\L}}}
\newcommand{\LXT}{\mathbb{L}_{\X \times T}}
\newcommand{\OT}{\otimes}

\newcommand{\RHom}{R\mathcal{H}om}
\newcommand{\Ext}{\operatorname{Ext}}
\newcommand{\RP}{R\pi_{*}}

\renewcommand{\c}{\operatorname{c}}
\renewcommand{\L}{\mathcal{L}}
\newcommand{\DT}{\operatorname{DT}}

\def\<{\left\langle}
\def\>{\right\rangle}

\begin{document}

\title{On Donaldson-Thomas invariants of threefold stacks and gerbes}
\author{Amin Gholampour}
\address{Department of Mathematics\\ California Institute of Technology\\ Pasadena\\ CA 91125\\USA}
\email{agholamp@caltech.edu}
\author{Hsian-Hua Tseng}
\address{Department of Mathematics\\ Ohio State University\\ 100 Math Tower, 231 West 18th Ave.\\Columbus\\ OH 43210\\ USA}
\email{hhtseng@math.ohio-state.edu}

\date{\today}

\begin{abstract}
We present a construction of Donaldson-Thomas invariants for three-dimensional projective Calabi-Yau Deligne-Mumford stacks.  We also study the structure of these invariants for \'etale gerbes over such stacks.
\end{abstract}

\maketitle
\section{Introduction}\label{sec:intro}
We work over the field of complex numbers throughout the paper. Let $\X$ be a smooth proper Deligne-Mumford (DM) stack with projective coarse moduli space $X$. Gromov-Witten (GW) theory, which roughly speaking concerns integrations against virtual fundamental classes of moduli spaces of twisted stable maps, is by now well-established \cite{cr}, \cite{agv1}, \cite{agv2}, and has been an area of active research recently.

For $3$-dimensional smooth projective varieties the so-called Donaldson-Thomas (DT) theory is constructed in \cite{Thomas-Casson}. An important special case is when DT theory concerns integration against virtual fundamental classes of the moduli spaces of torsion free, rank $1$ sheaves with trivial determinants. It has been conjectured \cite{mnop1, mnop2}, and proven in some cases, that this gives an equivalent theory to the GW theory of the ambient 3-fold.

The first goal of this paper is to extend the construction of DT invariants to DM stacks. This is done in Section \ref{section:DT-inv_construction}. Our construction is parallel to that of \cite{Thomas-Casson}, and uses the moduli spaces of stable sheaves on a smooth projective DM stack $\X$, recently constructed by Nironi \cite{Nironi-Sheaves}. More precisely, we show that moduli spaces of stable sheaves on a $3$-dimensional DM stack $\X$ with trivial canonical bundle (i.e. Calabi-Yau) admit natural perfect obstruction theories, which we use to define invariants in case there are no strictly semistable sheaves.

For a given torsion element $[c] \in H^2_{et}(\X,\G_m)$ represented by a $2$-cocycle $c$, we will use the moduli space of stable $c$-twisted sheaves on $\X$ (see \cite{Caldararu-thesis,Lieblich,Nironi-Sheaves}). This is a connected component of the moduli space of stable sheaves on a $\G_m$-gerbe defined on $\X$ representing the element $[c]$ (see \cite[Appendix A]{Nironi-Sheaves}). The perfect obstruction theory on the moduli space of the stable sheaves (if exists) induces a perfect obstruction theory on the moduli space of stable $c$-twisted sheaves.

As an application of the construction in this paper, we study the DT invariants of a $G$-gerbe $\Y$ over a $3$-dimensional Calabi-Yau stack $\X$, where $G$ is a finite group. This is the content of Section \ref{section:DT-gerbe}. Our study is motivated by the physical conjecture in \cite{HHPS} which states that ``conformal field theories on the $G$-gerbe $\Y$ are equivalent to conformal field theories on a dual space $\widehat{\Y}$ {\em twisted by} a B-field $c$''. The construction of the dual space $\widehat{\Y}$ is explained in Definition \ref{defn:dual-space}. Mathematically, the B-field $c$ is a $\mathbb{G}_m$-valued cocycle on $\widehat{\Y}$. Its definition is explained in Section \ref{subsection:gerbe-basics}. The theme of Section \ref{section:DT-gerbe} is a comparison between DT invariants of the gerbe $\Y$ and DT invariants of the dual $(\widehat{\Y}, c)$. Our Proposition \ref{prop:DT-decomposition} can be interpreted as DT-theoretic version of the physics conjecture.

In the presence of strictly semistable sheaves, moduli spaces of stable sheaves are not proper. In this situation it is desirable to extend the construction of {\em generalized DT invariants} \cite{JS} to our setting. We plan to present this in a later revision.

\subsection*{Acknowledgment}
We thank A. Caldararu and M. Lieblich for discussion on twisted sheaves and gerbes, and X. Tang for helpful comments. H.-H. T. is supported in part by NSF grant DMS-0757722.

\section{DT invariants for DM-stacks}\label{section:DT-inv_construction}
\subsection{Review of Nironi's construction}
In this section we construct a perfect obstruction theory for the moduli space of stable sheaves over a special class of 3-dimensional DM stacks that is used in this paper. The construction is similar to the case of smooth varieties (see \cite{Thomas-Casson}). This will allow us to define DT invariants for such stacks in cases where the moduli schemes are projective.

Let $\X$ be a DM stack with a projective moduli scheme $X$. We denote by $\c:\X \to X$ the ``coarsening map'' (i.e. the natural map from the stack $\X$ to its coarse moduli scheme $X$). We further assume that $\X$ is equipped with a generating sheaf $\gen$ in the sense of 
\cite{Olsson-Starr, Nironi-Sheaves}. By definition $\gen$ is a locally free sheaf on $\X$ whose fiber over any geometric point of $x\in \X$ contains the regular representation of the stabilizer group at $x$. Throughout the paper we fix a choice of a generating sheaf $\gen$ of $\X$ and a polarization $\O_X(1)$ on $X$. Following \cite{Kresch} and \cite[Definition 2.20]{Nironi-Sheaves}, we call $\X$  {\em{projective}} if it satisfies these conditions\footnote{A DM stack $\X$ is projective if and only if it is a tame separated global quotient with a projective moduli scheme (see \cite{Kresch} and \cite[Theorem 2.21]{Nironi-Sheaves}).}.

In \cite{Nironi-Sheaves} Nironi constructs the moduli space of semistable coherent sheaves on projective smooth DM stacks. We review part of his construction briefly and refer the reader to \cite{Nironi-Sheaves} for details. The main difference with the case of coherent sheaves on schemes is that stability condition now depends on $\gen$ as well as $\O_X(1)$. More precisely, for a coherent sheaf $\F$ on $\X$ stability is defined with respect to the Hilbert polynomial $$P_{\F}(m):=\chi(\X,\F\otimes \gen^\vee \otimes \operatorname{c}^*\O_X(1)^{\otimes m}).$$  Let $p_{\F}$ be the monic polynomial obtained by dividing $P_{\F}$ by the coefficient of the leading term. $p_\F$ is the {\em reduced} Hilbert polynomial of $\F$. A pure sheaf $\F$ on $\X$ is called {\em semistable} if $p_{\F'}\le p_{\F}$ for any proper subsheaf $\F'\subset \F$. $\F$ is called {\em stable} if the inequality is always strict.

Let $P\in \mathbb{Q}[z]$, and let $$\M(\X,P)=\M(\X,\gen,\O_X(1),P)$$ be the moduli stack of stable torsion free sheaves $\F$ on $\X$ with $P_{\F}=P$. Nironi constructs $\M(\X,P)$ as a quotient stack $[Q/GL(N)]$ where $Q$ is an appropriate subscheme of a Quot scheme on $\X$ (see \cite{Olsson-Star}). He shows that $\M(\X,P)$ is a $\G_m$-gerbe over a quasi projective moduli scheme $M(\X,P)$. Moreover $M(\X,P)$ is shown to be a geometric quotient of $Q$, and GIT techniques provides a natural compactification of $M(\X,P)$, parameterizing the $S$-equivalence classes of semistable sheaves on $\X$ (see \cite[Theorems 6.20-23]{Nironi-Sheaves}).

\subsection{Obstruction theory and DT inariants}

Let $\L$ be a line bundle on $\X$, and denote by $$\M(\X,P,\L)\subset \M(\X,P)$$ the moduli substack of stable torsion free sheaves with fixed determinant $\L$. We denote by $M(\X,P,\L)$ the corresponding coarse moduli scheme. By the construction of Nironi and discussion above, $\M(\X,P,\L)$ is the fine moduli stack of DM type\footnote{In fact $\M(\X,P,\L)$ is a $\mu_r$-gerbe over $M(\X,P,\L)$ where $r$ is the rank of the objects parameterized by $\M(\X,P)$.}. By the following proposition there exist perfect obstruction theories on $\M(\X,P,\L)$ and $M(\X,P)$, in the sense of \cite{Be-Fa}:

\begin{prop} \label{prop:perfect obstruction theory}
Suppose $\X$ is a smooth projective DM stack of dimension 3 satisfying $\omega_{\X}\cong \O_{\X}$. Then there exist natural perfect obstruction theories on $\M(\X,P,\L)$ and $M(X,P)$. Moreover, these obstruction theories are symmetric in the sense of \cite{beh}.
\end{prop}
\begin{proof}
 We first treat the case $\M_\L=\M(\X,P,\L)$. Let $$\U\to \XM$$ be the universal stable sheaf over $\XM$. For a closed point $m\in \M$, let $\U_m\to \X$ be the stable sheaf parameterized by $m$. Note that we have the follow equation on trace-less $Ext$ groups $\Ext^\bullet_\X(-,-)_0$:
\begin{equation} \label{equ:Ext0=Ext3=0}
\Ext^3_{\X}(\U_m,\U_m)_0=\Ext^0_{\X}(\U_m,\U_m)_0=0.
\end{equation}
The first equality follows from Serre Duality for DM stacks (see \cite[Theorem 1.32]{Nironi-duality}) and the assumption that $\omega_\X$ is trivial. The second equality is because of the stability of $\U_m$.

The construction of the obstruction theory for $\M_\L$ is similar to the case of the moduli space of sheaves on smooth Calabi-Yau threefolds (see \cite{Thomas-Casson}). In what follows, $\mathbb{L}_{\square}$ denotes the cotangent complex of a DM stack (see \cite{Illusie}).

Let $\pi: \XM\to \M_\L$ be the projection. Composing the {\em Atiyah class} (see \cite[IV 2.3.6.2]{Illusie}) $$\U \to \LXM \OT \U [1]$$ with the natural projection $\LXM\to \pi^*\LM$ gives $$\U \to \pi^*\LM[1]\OT\U.$$ This gives a morphism $$\RHom(\U,\U)\to \pi^*\LM[1].$$ Since $\pi$ is smooth of relative dimension 3, tensoring both sides by $\O[2]$ (note that $\omega_\pi=\pi^*\omega_{\X}\cong\O$) yields a morphism $$\RHom(\U,\U[2]) \to \pi^!\LM.$$ By duality theorem (see \cite[Corollary 1.22 and Theorem 1.32]{Nironi-duality}) this gives a morphism $$\RP \RHom(\U,\U[2])\to \LM$$ which, after restricting the left hand side to its traceless part, gives a morphism $$\phi:\E=\RP \RHom(\U,\U[2])_0\to \LM.$$  In what follows we show that $(\E,\phi)$ is a perfect obstruction theory.

First note that $\E$ is perfect of perfect amplitude contained in $[-1,0]$. This is true because of (\ref{equ:Ext0=Ext3=0}) (see \cite[Lemma 4.2]{Thomas-Huy}).

Next we need to show that $(\E,\phi)$ is an obstruction theory. Suppose  $g:T \to \M$ is a morphism from a scheme $T$ and $T \to \T$ is an extension by a square-zero ideal $I$, then we need to show that the obstruction to extending  $g$ to $\T$ is a class $w\in \Ext^1_{T}(Lg^*\E, I)$ obtained by composing $Lg^*\phi$ with the natural maps $Lg^*\LM\to \LT \to I[1]$.  To show $(\E,\phi)$ is an obstruction theory it suffices to check the following criterion (see \cite[Theorem 4.5]{Be-Fa}):
\begin{claim}
$w=0$ if and only if there exists an extension $\g:\T\to \M$, and if nonempty the set of all such $\g$ makes a torsor over $\Ext^0_{T}(Lg^*\E, I)$.
\end{claim}
We now prove this Claim. Let $f=(id,g)$ and $\f=(id,\g)$. Consider the following diagram:
\begin{equation*}
\begin{CD}
\XT@>f>>\XM@>q>>\X\\
@V{p}VV @V{\pi}VV\\
T@>g>>\M_\L.
\end{CD}
\end{equation*}
By standard arguments one obtains a natural identification
\begin{equation}\label{equ:identification}
\Ext^i_{T}(Lg^*\E, I) \cong \Ext^{i+1}_{\XT}(f^*\U,f^*\U \OT p^*I)_0.
\end{equation}
Because $\M_\L$ is a fine moduli space, deforming $g$ to $\g$ is equivalent to deforming $f^*\U$ to $\f^*\U$. The obstruction to the latter, denoted by $w'$, is obtained by the composition $$f^*\U\to \LXT \OT f^*\U[1] \to p^*I\OT f^*\U[2]$$ and then restricting to the traceless part. The first map is the Atiyah class $at(f^*\U)$ and the second one is induced from the natural map $\LXT\to p^*I[1]$. This is true because of \cite[Proposition IV.3.1.8]{Illusie} and the fact that we have fixed the determinant. The reason for restricting to the traceless part is that line bundles on $\X$ are unobstructed, and as in the case of sheaves on schemes (see \cite{Thomas-Casson}), one can show that the trace of the obstruction class of a sheaf $\F$ on a smooth DM stack is the obstruction class of $\det(\F)$.

For a similar reason and by using (\ref{equ:identification}), if $w'=0$ the set of all deformations is a torsor over $\Ext^0_{T}(Lg^*\E, I).$ So it remains to show that $w$ is mapped to $w'$ under (\ref{equ:identification}). We showed that $\phi:\E \to \LM$ arises from the Atiyah class $at(\U)$. Following exactly the same steps one can show that $Lg^*\E \to Lg^*\LM$ arises from the Atiyah class $at(f^*\U)$. This means that the class $w$, which is the composition
$$Lg^*\E \overset{\phi}{\longrightarrow} Lg^*\LM \to \LT \to I[1],$$
gives rise to
$$f^*\U\to \LXT \OT f^*\U[1] \to p^*I\OT f^*\U[2],$$ which is what we need. This finishes the proof of the Claim, and the construction of the perfect obstruction theory on $\M(\X,P,\L)$.

To construct an obstruction theory on $M=M(\X,P)$, we just make the following modifications to the construction given above. Firstly, $\U$ is replaced by the universal twisted sheaf on $M \times \X$ denoted by $U$ (see \cite[Section 3.3]{Caldararu-thesis}). Secondly, the natural candidate $$\RP \RHom(U,U[2])\cong (\RP\RHom(U,U))^\vee[-1]$$ for the obstruction theory is not perfect, where now $\pi: \X \times M \to M$ is the projection. However, repeating the arguments in \cite[Section 4.4]{Thomas-Huy}, it can be shown that the trimmed complex $$(\tau^{[1,2]}\RP\RHom(U,U))^\vee[-1]$$ gives rise to a perfect obstruction theory on $M(\X,P)$.

The symmetry of the obstruction theories on $\M(\X,P,\L)$ and $M(\X,P)$ follows easily from Serre duality and the Calabi-Yau condition $\omega_{\X}\cong \O_{\X}$.
\end{proof}

If $\X$ is as in Proposition \ref{prop:perfect obstruction theory}, then by the symmetry property of the obstruction theories the expected dimensions of $\M(\X,P,\L)$ and $M(\X,P)$ are 0. By Proposition \ref{prop:perfect obstruction theory} and \cite{Be-Fa} we have
\begin{prop}
Suppose $\X$ is a smooth projective DM stack of dimension 3 satisfying $\omega_{\X}\cong \O_{\X}$. Then $\M(\X,P,\L)$ and $M(\X,P)$ carry virtual 0-cycles denoted by $[\M(\X,P,\L)]^{vir}$ and $[M(\X,P)]^{vir}$.
\end{prop}

Let $\nu_{\M}$ (respectively, $\nu_{\M_\L}$) be the Behrend's function (\cite{beh,JS}) defined on $\M(\X,P)$ (respectively, $\M(\X,P,\L)$). Then we define the corresponding DT invariants as follows
\begin{defn}
Let $\X$ be as in Proposition \ref{prop:perfect obstruction theory}, $P\in \mathbb{Q}[z]$, and $\L$ a line bundle on $\X$. Then we define the {\em Donaldson-Thomas invariants}\footnote{These invariants depend on the choices of $\gen$ and $\O_X(1)$, however this dependence is suppressed in our notation.} of $\X$ corresponding to $P$ (and $\L$) as the weighted Euler characteristics
$$\DT(\X,P)=-\chi^{na}(\M(\X,P),\nu_\M).$$
$$\DT(\X,P,\L)=\chi(\M(\X,P,\L),\nu_{\M_\L}).$$
Here $\chi^{na}$ is the na\"{i}ve Euler characteristic defined for Artin stacks (see \cite[Definition 2.3]{JS} and $\chi$ denotes the Euler characteristic of DM stacks.

\end{defn}
\begin{remark}
\hfill
\begin{enumerate}
\item Let $r$ be the rank of the objects parameterized by $\M(\X,P)$, and let $\nu_M$ (respectively $\nu_{M_\L}$) be the Behrend's function for $M(\X,P)$ (respectively, for $M(\X,P,\L)$) then by the properties of the Behrend's function and Weighted Euler characteristic (see \cite{beh,JS}), we have
    $$\DT(\X,P)=\chi(M(\X,P),\nu_M),$$ and
    $$\DT(\X,P,\L)=\frac{1}{r}\chi(M(\X,P,\L),\nu_{M_\L}).$$
\item Suppose that there are no strictly semistable sheaves $\F$ on $\X$ satisfying $P_{\F}=P$.
It is known \cite{Nironi-Sheaves} that in this case $M(\X,P)$ and $\M(\X,P,\L)$ are proper, and hence the virtual classes $[M(\X,P)]^{vir}$ and $[\M(\X,P,\L)]^{vir}$ can be integrated. By \cite[Theorem 4.18]{beh} and \cite[Remark 5.14]{JS}
$$\DT(\X,P,\L)=\deg\left([\M(\X,P,\L)]^{vir}\right),$$ and
$$\DT(\X,P)=\deg\left([M(\X,P)]^{vir}\right).$$
\item If there are strictly semistable sheaves, then $\DT(\X,P)$ and $\DT(\X,P,\L)$ are not in general deformation invariant. One way to fix this is to extend the construction of {\em generalized Donaldson-Thomas invariants} \cite{JS} to this setting. This will be explored in a later revision.
\end{enumerate}
\end{remark}

\section{A decomposition result for DT invariants on gerbes}\label{section:DT-gerbe}

The purpose of this Section is to study DT invariants of \'etale gerbes.

\subsection{\'Etale gerbes}\label{subsection:gerbe-basics}
We begin with a review of some basic notions of \'etale gerbes and the construction of their duals. Let $G$ be a finite group\footnote{$G$ is viewed as a finite group scheme over $\text{Spec}\, \mathbb{C}$.}. Let $\X$ be a smooth projective Deligne-Mumford stack with coarse moduli scheme $X$. Let $BG$ denote that stack of $G$-torsors.
\begin{definition}[see e.g. \cite{EHKV}, Definition 3.1]
A $G$-gerbe over $\X$ is a Deligne-Mumford stack $\Y$ together with a morphism $\Y\to \X$ such that there exists a faithfully flat, locally of finite presentation, map $X'\to \X$ such that $\Y\times_\X X'\simeq BG\times X'$.
\end{definition}

In this way one can view $BG$ as a $G$-gerbe over a point.

Let $Out(G)$ denote the group of outer automorphisms of $G$. By definition, $Out(G)$ is the quotient of the group $Aut(G)$ of automorphisms of $G$ by the normal subgroup $Inn(G)$ of inner automorphisms of $G$, $$Out(G)=Aut(G)/Inn(G).$$

Given a $G$-gerbe $\Y\to \X$, there is a naturally defined $Out(G)$-bundle $\overline{\Y}\to \X$, called the {\em band}. See \cite{EHKV}, Definition 3.3 for a detailed definition. We say that the $G$-gerbe $\Y\to \X$ has {\em trivial band} if the $Out(G)$-bundle $\overline{\Y}\to \X$ is endowed with a section (hence is trivialized by this section).

Let $\widehat{G}$ denote the {\em set} of isomorphism classes of irreducible representations of $G$. Note that $\widehat{G}$ is a finite set, the cardinality of $\widehat{G}$ coincides with the number of conjugacy classes of $G$. We may also view $\widehat{G}$ as a disjoint union of points.

Let $\rho: G\to End(V_\rho)$ be an irreducible representation of $G$, and $\phi\in Aut(G)$. Then the composite $$\rho\circ \phi^{-1}: G\to End(V_\rho)$$ is an irreducible representation of $G$. It is easy to see that this induces an action of $Out(G)$ on $\widehat{G}$. Note that the isomorphism class $[1_{tr}]$ of the $1$-dimensional trivial representation of $G$ is fixed by this $Out(G)$ action.

\begin{definition}[see \cite{HHPS}]\label{defn:dual-space}
Define $$\widehat{\Y}:=\overline{\Y}\times_{Out(G)} \widehat{G}.$$
There is a natural map $\widehat{\Y}\to \X$ induced from the map $\overline{\Y}\to \X$.
\end{definition}

\begin{remark}
\hfill
\begin{enumerate}
\item
The morphism $\widehat{\Y}\to \X$ is finite and \'etale. The stack $\widehat{\Y}$ is {\em disconnected}.

\item
If the $G$-gerbe $\widehat{\Y}\to \X$ has trivial band, then $\widehat{\Y}$ is a disjoint union of several copies of $\X$, and the map $\widehat{\Y}\to \X$ restricts to the identity on each copy.

\end{enumerate}
\end{remark}

For each isomorphism class $[\rho]\in \widehat{G}$ we fix a representation $\rho: G\to End(V_\rho)$ in this class. To each $(x, [\rho])\in \widehat{\Y}$ we assign the vector space $V_\rho$. This defines a family of vector spaces over $\widehat{\Y}$, which is in general {\em not} a vector bundle over $\widehat{\Y}$. The obstruction to find a vector bundle over $\widehat{\Y}$ with fiber over $(x, [\rho])$ being $V_\rho$ is a $\mathbb{G}_m$-valued $2$-cocycle on $\widehat{\Y}$, whose {\em inverse} is denoted by $c$.

As observed in \cite{HHPS}, the cocycle $c$ is locally constant, and $c$ represents a {\em torsion} class in the cohomology $H^2_{et}(\widehat{\Y},\mathbb{G}_m)$.

Another way to understand the cocycle $c$ is the following. The failure for $V_\rho$'s to form a vector bundle is due to the fact that they glue {\em up to scalars}. In other words $V_\rho$'s glue to a {\em twisted sheaf} (see e.g. \cite{Caldararu-thesis}). This twisted sheaf is equivalent (see \cite{Lieblich}) to a sheaf on a $\mathbb{G}_m$-gerbe over $\widehat{\Y}$. This $\mathbb{G}_m$-gerbe turns out to be flat and the inverse of its class, which is an element in $H^2_{et}(\widehat{\Y}, \mathbb{G}_m)$, is represented by the $2$-cocycle $c$.

\subsection{Equivalence} \label{subsection:Equivalence}
In what follows we will be concerned with sheaves on gerbes and a decomposition statement about DT invariants for Calabi-Yau gerbes, which is inspired by \cite{HHPS}.

We continue to use the notation in the previous section. Let $\Y\to \X$ be a $G$-gerbe over a smooth projective DM stack $\X$ and let $\widehat{\Y}\to \X$ and $c$ be as constructed before. Note that $\Y$ is also a smooth projective DM stack, and the coarse moduli space of $\Y$ is $X$. Let $c_\Y: \Y\to X$ be the coarsening map. By construction, $\widehat{\Y}$ is also smooth and projective. Let $c_{\widehat{\Y}}: \widehat{\Y}\to \widehat{Y}$ denote its coarsening map, and let $\pi_{\widehat{Y}}: \widehat{Y}\to X$ be the map between coarse moduli spaces induced by $\widehat{\Y}\to \X$.

Fix an ample line bundle $\sO_X(1)$ of $\X$. Note that the pull-back $\pi_{\widehat{Y}}^*\sO_X(1)$ is an ample line bundle of $\widehat{Y}$.

Let $Coh(\Y)$ denote the abelian category of coherent sheaves on $\Y$, and $Coh(\widehat{\Y}, c)$ the abelian category of coherent $c$-{\em twisted} sheaves on $\widehat{\Y}$. We refer to \cite{Caldararu-thesis} and \cite{Lieblich} for detailed discussions on the theory of twisted sheaves. The following result is proven in \cite{Tang-Tseng}.

\begin{theorem}[X. Tang-H.-H. Tseng]
There is natural functor
\begin{equation}\label{equivalence}
F: Coh(\Y)\to Coh(\widehat{\Y}, c),
\end{equation}
which is an equivalence of abelian categories.
\end{theorem}
The construction of this functor $F$ is rather involved. Details can be found in \cite{Tang-Tseng}. Roughly speaking, the inverse functor $Coh(\widehat{\Y}, c) \to Coh(\Y)$ can be understood as taking $(-)\otimes \mathcal{V}_\rho$, where $\mathcal{V}_{\rho}$ is the aforementioned $c^{-1}$-twisted sheaf with fibers $V_\rho$.

As noted above, $\widehat{\Y}$ is disconnected. Let $\widehat{\Y}=\coprod_{i\in \sI} \widehat{\Y}_i$ be the decomposition of $\widehat{\Y}$ into connected components, and let $c_i$ be the $2$-cocycle on $\widehat{\Y}_i$ obtained by restriction of $c$. By definition, we have $$Coh(\widehat{\Y}, c)=\oplus_{i\in \sI}Coh(\widehat{\Y}_i, c_i).$$ Consequently there is a decomposition of $K$-groups $$K(Coh(\widehat{\Y}, c))=\oplus_{i\in \sI} K(Coh(\widehat{\Y}_i, c_i)).$$
On the other hand, $K(\Y)=K(Coh(\Y))=K(Coh(\widehat{\Y}, c))$, therefore we get a decomposition of $K(\Y)$:
\begin{equation}\label{decomp_K_group}
K(\Y)=\oplus_{i\in \sI} K_i, \quad K_i:=K(Coh(\widehat{\Y}_i, c_i)).
\end{equation}

Given $\F\in Coh(\Y)$, we write $$F(\F)=\oplus_{i\in \sI} F(\F)_i,\quad F(\F)_i\in Coh(\widehat{\Y}_i, c_i).$$ Since (\ref{equivalence}) is an equivalence of abelian categories, it preserves exact sequences. Hence for $\F\in Coh(\Y)$ and a subsheaf $\F'\subset \F$, the components $F(\F')_i$ are subsheaves of $F(\F)_i$. Also for $\F_1, \F_2 \in Coh(\Y)$ we have the equality on $Hom$ spaces
\begin{equation}\label{equality_Hom}
Hom_{Coh(\Y)}(\F_1, \F_2)=\oplus_{i\in \sI}Hom_{Coh(\widehat{\Y}_i, c_i)}(F(\F_1)_i, F(\F_2)_i).
\end{equation}
By the construction of the equivalence (\ref{equivalence}), it is easy to check that if $\V$ is a generating sheaf of $\Y$, then $F(\V)$ is a generating $c$-twisted sheaf of $\widehat{\Y}$. Hence $F(\V)_i\in Coh(\widehat{\Y}_i, c_i)$ is a generating $c_i$-twisted sheaf of $\widehat{\Y}_i$. Since $G$ acts trivially on $c_\Y^*\sO_X(1)$, the construction of the equivalence (\ref{equivalence}) implies that $$F(\F\otimes c_\Y^*\sO_X(1)^{\otimes m})=F(\F)\otimes c_{\widehat{\Y}}^*\pi_{\widehat{Y}}^*\sO_X(1)^{\otimes m}.$$

From now on, fix a generating sheaf $\gen$ on $\X$ and an ample line bundle $\sO_X(1)$ on $X$. Also fix the generating $c$-twisted sheaf $F(\gen)$ on $\widehat{\Y}$ and an ample line bundle $\pi_{\widehat{Y}}^*\sO_X(1)$ on $\widehat{Y}$. With these choices it follows that the Hilbert polynomial $P_\F$ of $\F$ coincides with the Hilbert polynomial $P_{F(\F)}$ of $F(\F)$. More precisely,
\begin{equation}\label{eqn_HilbPoly}
P_\F=P_{F(\F)}=\sum_{i\in \sI} P_{F(\F)_i}.
\end{equation}

\subsection{Invariants}\label{sect:invariants}
Let $$C(\Y)=\{[\F]\in K(Y) | 0\neq\F \in Coh(\Y)\}$$ be the positive cone in $K(\Y)$. Then $C(\Y)=\oplus_{i\in \sI} C_i$ corresponding to the decomposition (\ref{decomp_K_group}). Let $k\in C(\Y)$, and let $\M(\Y, k)$ be the moduli stack of stable torsion free sheaves on $\Y$ of class $k$. It is evident that $\M(\Y, k)$ is a component of the moduli stack of stable torsion free sheaves with fixed Hilbert polynomials. Suppose further that $k\in C_i$ in the decomposition above. Note that for any $\F\in Coh(\Y)$ of class $k$, we have $F(\F)=F(\F)_i\in Coh(\widehat{\Y}, c_i)$, namely
\begin{equation}\label{vanishing_component}
F(\F)_j=0\quad \text{for } j\neq i.
\end{equation}
By (\ref{eqn_HilbPoly}) and (\ref{vanishing_component}) we have the following relations between (reduced) Hilbert polynomials:
$$P_\F=P_{F(\F)_i}, \quad p_\F= p_{F(\F)_i}.$$
Consequently $\F$ is (semi)stable if and only if $F(\F)_i$ is (semi)stable.

 Therefore the equivalence (\ref{equivalence}) yields a set-theoretic bijection
$$\M(\Y, k)\to \M((\widehat{\Y}_i, c_i), k), \quad [\F]\mapsto [F(\F)_i].$$ Here $\M((\widehat{\Y}_i, c_i), k)$ denotes the moduli of semistable $c_i$-twisted sheaves on $\widehat{\Y}_i$ of class $k$. As mentioned in Section \ref{sec:intro}, $\M((\widehat{\Y}_i, c_i), k)$ is realized as a connected component of the certain moduli space of stable sheaves on a $\mu_N$-gerbe over $\widehat{\Y}_i$ for $N\gg 0$, and hence Nironi's construction applies (see \cite[Appendix A]{Nironi-Sheaves}).

\begin{proposition}\label{decomp_moduli_of_sheaf}
There is an isomorphism of stacks $$\M(\Y, k)\simeq \M((\widehat{\Y}_i, c_i), k).$$
\end{proposition}
\begin{proof}
This is proven by checking that deformation theory on both sides agree to all order, using (\ref{equality_Hom}). Alternatively the isomorphism can be constructed as follows. Clearly the product $\M(\Y,k)\times \Y$ is a $G$-gerbe over $\M(\Y, k)\times \X$. By construction we see that the dual of this $G$-gerbe is $$\widehat{\M(\Y,k)\times \Y}=\M(\Y,k)\times \widehat{\Y}.$$ Moreover the $2$-cocycle in this case is the pull-back of $c$ via the projection $\M(\Y,k)\times \widehat{\Y}\to \widehat{\Y}$. Let $\U\to \M(\Y, k)\times \widehat{\Y}$ be the universal stable sheaf. Note that there exists an equivalence (\ref{equivalence}) for any $G$-gerbe. Applying such an equivalence to the sheaf $\U$ over the $G$-gerbe $\M(\Y,k)\times \Y$, we obtain a twisted sheaf $F(\U)$ over $\M(\Y, k)\times \widehat{\Y}$. It is easy to check that $F(\U)$ is a family over $\M(\Y,k)$ of $c_i$-twisted stable sheaves with class $k$. This defines a morphism $\M(\Y,k)\to \M((\widehat{\Y}_i,c_i),k)$. The inverse morphism can be defined in a similar fashion by using an inverse of (\ref{equivalence}).
\end{proof}
Now suppose in addition that $\X$ is a Calabi-Yau Deligne-Mumford stack of dimension $3$. Hence both $\Y$ and $\widehat{\Y}$ are Calabi-Yau of dimension $3$. Then we can define DT invariants $\DT(\Y, k)$ and $\DT((\widehat{\Y}_i, c_i), k_i)$ as in Section \ref{sect:invariants}, and by Proposition \ref{decomp_moduli_of_sheaf} we have
\begin{proposition}
$$\DT(\Y, k)=\DT((\widehat{\Y}_i, c_i), k).$$
\end{proposition}
\begin{remark}
One can see directly, using (\ref{equality_Hom}), that the $2$-term perfect obstruction theory associated to $\M(\Y, k)$ is mapped to the one on $\M((\widehat{\Y}_i, c_i), k)$. This gives an alternative proof for the proposition above in cases where $k\in C(\Y)$ is such that for sheaves of class $k$ semistability and stability coincide.
\end{remark}

\subsection{Decomposition}
Using the notation in Sections \ref{subsection:gerbe-basics}-\ref{subsection:Equivalence}, let $\Y$ be a $G$-gerbe over a smooth projective DM stack $\X$, and Let $k\in C(\Y)$. Suppose $k_i$ is the $C_i$-component of $k$ in the decomposition $C(\Y)=\oplus_{i\in \sI} C_i$ induced from (\ref{decomp_K_group}), $$k=\sum_i k_i, \quad k_i\in C_i.$$
Let $\M(\Y, k)$ be the moduli of stable sheaves on $\Y$ of class $k$. It is evident that $\M(\Y, k)$ is a component of the moduli of stable sheaves with fixed Hilbert polynomials. We make the following assumption
\begin{assump}\label{semistability_assump}
\hfill
A sheaf $\F\in Coh(\Y)$ of class $k$ is (semi)stable if and only if $F(\F)_i\in Coh(\widehat{\Y}_i, c_i)$ is (semi)stable for all $i\in \sI$.
\end{assump}

\begin{remark}
Naively Assumption \ref{semistability_assump} should follow from the equality (\ref{eqn_HilbPoly}) of Hilbert polynomials.  We are unable to deduce Assumption \ref{semistability_assump} from (\ref{eqn_HilbPoly}) because the notion of (semi)stability is defined using the {\em reduced} Hilbert polynomial, and it is not clear to us whether (\ref{eqn_HilbPoly}) holds for reduced Hilbert polynomials or not.
\end{remark}

In view of Assumption \ref{semistability_assump} the equivalence (\ref{equivalence}) yields a set-theoretic bijection
$$\M(\Y, k)\to \prod_{i\in \sI}\M((\widehat{\Y}_i, c_i), k_i), \quad [\F]\mapsto ([F(\F)_i])_{i\in \sI}.$$ Here $\M((\widehat{\Y}_i, c_i), k_i)$ denotes the moduli of semistable $c_i$-twisted sheaves on $\widehat{\Y}_i$ of class $k_i$.

\begin{proposition}\label{decomp_moduli_of_sheaf2}
Assume Assumption \ref{semistability_assump}, then there is an isomorphism of stacks $$\M(\Y, k)\simeq \prod_{i\in \sI}\M((\widehat{\Y}_i, c_i), k_i).$$
\end{proposition}
\begin{proof}
This is proved by the same arguments as in the proof of Proposition \ref{decomp_moduli_of_sheaf}
\end{proof}

Now suppose in addition that $\X$ is a Calabi-Yau DM stack of dimension $3$. Then Proposition \ref{decomp_moduli_of_sheaf2} and the multiplicativity of the Behrend's function \cite[Proposition 1.5]{beh} we have the following relation among the DT invariants:
\begin{proposition}\label{prop:DT-decomposition}
Assume Assumption \ref{semistability_assump}, then $$\DT(\Y, k)=\prod_{i\in \sI}\DT((\widehat{\Y}_i, c_i), k_i).$$
\end{proposition}

This is our decomposition statement for the DT invariants of the gerbe $\Y$.

\begin{remark} By the same discussion as in Section \ref{sect:invariants}, the perfect obstruction theory associated to $\M(\Y, k)$ is mapped to that on $\prod_{i\in \sI}\M((\widehat{\Y}_i, c_i), k_i)$. This observation gives an alternative proof of Proposition \ref{prop:DT-decomposition}  in the cases semistability and stability coincide for all sheaves of class $k$, and for all $c_i$-twisted sheaves of class $k_i$, $i\in \sI$. \end{remark}


\end{document}